\documentclass{article}
\usepackage[utf8x]{inputenc}
\usepackage{amsthm}
\usepackage{amssymb}
\usepackage{amsmath}
\usepackage{amsfonts}
\usepackage{latexsym}
\usepackage{graphicx}
\usepackage[inline]{enumitem}
\usepackage{fullpage}
\usepackage{multirow}
\usepackage{pgfplots}
\usepackage[linesnumbered,ruled,vlined]{algorithm2e}

\newcommand\ForAuthors[1]
 {\par\smallskip                     
  \begin{center}
   \fbox
   {\parbox{0.9\linewidth}
    {\raggedright\sc--- #1}
   }
  \end{center}
  \par\smallskip                     %
 }
 \newcommand\comment[1]{}
\def\rr{\mathbb R}
\def\cc{\mathbb C}
\def\cd{\cc^d}
\def\nn{\mathbb N}
\def\rd{\rr^d}
\def\co{c_0(\rr)}
\def\gse{\Delta^{\geqslant}}
\def\gs{\Delta^{>}}
\def\agmx{\operatorname{Argmax}}
\def\inte{\operatorname{int}}
\def\poseq{\mathrm{Pos}^{\geqslant}}
\def\pos{\mathrm{Pos}^{>}}
\def\bigkse{\mathbf{K}^\geqslant}
\def\bigks{\mathbf{K}^>}
\def\ks{k^{>}}
\def\kse{k^{\geqslant}}
\def\nuopt{\nu_{\rm opt}}
\def\kopt{k_{\rm opt}}
\def\xopt{x_{\rm opt}}
\def\xyopt{(x_{\rm opt},y_{\rm opt})}
\def\xin{X^\mathrm{in}}
\def\rea {\mathcal{R}}
\def\norm#1{\| #1\|}
\def\lmax#1{\lambda_{\rm max}(#1)}
\def\vdiag{V_{\rm diag}}
\def\kdiag{\mathbf{K}^{\rm diag}}
\def\mink{\mathbf{K}}
\def\mus#1{\mu\left(#1\right)}
\def\Idd{\operatorname{Id}}
\def\btilde{\tilde{b}}
\newenvironment{psmallmatrix}
  {\left(\begin{smallmatrix}}
  {\end{smallmatrix}\right)}
\newtheorem{proposition}{Proposition}
\newtheorem{assumption}{Assumption}
\newtheorem{example}{Example}
\newtheorem{lemma}{Lemma}
\newtheorem{corollary}{Corollary}
\newtheorem{theorem}{Theorem}
\newtheorem{remark}{Remark}
\SetKwInOut{Input}{Input}
\SetKwInOut{Output}{Output}
\SetKwInOut{init}{Initialisation}
\begin{document}
\title{Quadratic Maximization over the Reachable Values Set of a Convergent Discrete-time Affine System :\\
The diagonalizable case}


\author{Assalé Adjé\thanks{This article has started when the author benefited from the support of the FMJH "Program Gaspard Monge for optimization and operations research and their interactions with data science", and from the support from EDF.}\\
Laboratoire de Mathématiques et Physique (LAMPS)\\
			  Université de Perpignan Via Domitia, Perpignan, France\\
              assale.adje@univ-perp.fr 
}


\date{}

\maketitle

\begin{abstract}
In this paper, we solve a maximization problem where the objective function is quadratic and convex or concave and the constraints set is the reachable value set of a convergent discrete-time affine system. Moreover, we assume that the matrix defining the system is diagonalizable. The difficulty of the problem lies in the infinite sequence to handle in the constraint set. Equivalently, the problem requires to solve an infinite number of quadratic programs. Therefore, the main idea is to extract a finite of them and to guarantee that the resolution of the extracted problems provides the optimal value and a maximizer for the initial problem. The number of quadratic programs to solve has to be the smallest possible. Actually, we construct a family of integers that over-approximate the exact number of quadratic programs to solve using basic ideas of linear algebra.  This family of integers is used in the final algorithm. A new computation of an integer of the family within the algorithm ensures a reduction of the number of loop iterations. The method proposed in the paper is illustrated on small academic examples. Finally, the algorithm is experimented on randomly generated instances of the problem.

\end{abstract}

\section{Introduction}
\subsection{Motivation}
In many situations, we are interested in the maximum of some objective function over the reachable values set of an uncontrolled discrete-time dynamical system. An interesting reachable value can be the one which penalizes the most the system with respect to a performance criteria. In particular, this type of optimization problem arises in verification of systems or programs~\cite{DBLP:journals/entcs/Adje15}. Indeed, to verify a program or a system consists in proving that the specifications, the rules for which the program or the system is designed for, are satisfied. For some numerical specifications, e.g. the absence of overflows, the verification problem remains to compute bounds over the possible values taken by each coordinate of the state-variable. Thus, it boils down to solve an optimization problem whose the set of constraints is exactly  all possible reachable values. Another example of verification problems is an input-output system in charge of the control of a mechanical structure. We have to take care about the outputs of the system. Indeed, the mechanical structure has physical constraints. Therefore, we have to check whether the outputs are suitable for the mechanics. Again, the closed-loop structure of whole system indicates a discrete-time dynamical system and the verification analysis can be reduced to the resolution of a maximization problem (e.g.~\cite{Saberi2000,7001601}).  Besides, those computations have to be done before the execution or the use of the program or the system. Classical methods to verify programs or systems are testing methods or simulations. Unfortunately, they are not capable of covering, in general, all possible situations. In consequence, testing methods and simulations are completed by static methods. Static means that the only usable data are the dynamics of the system or the structure of the program : the way of generating possible values not the values themselves. 





\subsection{Context}

In this paper, we are interested in solving the maximization of a possibly non-homogeneous quadratic function over the reachable value of a discrete-time affine system. More precisely, let us consider a $d\times d$ matrix $A$, a $\rd$-vector $b$ and a polytope $\xin$. We define the following discrete-time affine system starting from $x_0\in\xin$, for all $k\in\nn$, by :
\begin{equation}
\label{stateeq}
  x_{k+1}=Ax_k+b \enspace.
\end{equation}
The recurrence formulation of Eq.~\eqref{stateeq} can be replaced, for all $k\geq 1$ by
$x_k=A^k x_0+ \sum_{i=0}^{k-1} A^i b$,
where $A^l$ denotes the $l$-th power of the matrix $A$. This rewriting allows to associate to the affine system, its reachable value set $\rea$ i.e. 
 \begin{equation}
\label{reach}
\rea=\xin\cup\bigcup_{k\in \nn^*} \left(A^{k}(\xin)+\sum_{i=0}^{k-1} A^i b\right) \enspace.
\end{equation}
Finally, given a $d\times d$ symmetric matrix $Q$ and a $\rd$ vector $q$, we are interested in solving the following quadratic maximization problem :
 \begin{equation}
 \label{optpb}
 \sup_{x\in\rea} x^\intercal Q x +q^\intercal x \enspace.
 \end{equation}
Classical quadratic programming solvers cannot be used as a direct solution in our context. First, the set $\rea$ is not necessary closed and bounded. Second, we are not able to represent the set of constraints i.e., here, a reachable values set of a discrete-time affine system. This is essentially due to the fact that the feasible points are infinite sequences. However, the optimization problem depicted at Eq.~\eqref{optpb} can be viewed as sequence of standard quadratic maximization problems. Indeed, introducing the sequence of polytope:
\begin{equation}
\label{reachseq}
\rea_k=\left\{\begin{array}{lr}
\displaystyle{A^{k}(\xin)+\sum_{i=0}^{k-1} A^i b} & \text{ if } k>0\\
\xin & \text{ if } k=0
\end{array}\right.
\end{equation}
We can rewrite Problem~\eqref{optpb} as follows:
\begin{equation}
\label{eqaux}
\sup_{k\in\nn} \max_{x\in \rea_k} x^\intercal Q x+q^\intercal x
\end{equation}
Optimization problem~\eqref{eqaux} represents an \emph{infinite} sequence of linearly constrained problems with quadratic objective function. The difficulty is to extract a finite number of quadratic problems in Eq.~\eqref{eqaux} or similarly sequences of finite length in $\rea$. The length has to be uniform with respect to the polytopic initial set $\xin$. Furthermore, if we consider a too coarse uniform length, then we will drastically increase the number of evaluations of the quadratic objective function or the number of quadratic optimization problems to solve. The objective is thus to find the smallest uniform length possible. More formally, the main difficulty is to find, if it exists, the smallest possible integer $K$, such that:
\begin{equation}
\label{motive}
\sup_{x\in\rea} x^\intercal Q x +q^\intercal x =\max_{k=0,\ldots, K}\max_{x\in\rea_k} x^\intercal Q x+q^\intercal x
\end{equation}
Since $x\in\rea$ is completed determined by a starting point $x_0\in\xin$. Then if $x\in\rea$ is an optimal solution for Problem~\eqref{optpb} not belonging to $\xin$, then there exists $k\in\nn$ and $x_0\in\xin$ such that $x=A^k x_0+  \sum_{i=0}^{k-1} A^i b$. 
In consequence, an \emph{optimal solution} for Problem~\eqref{optpb} is a couple $(\kopt,\xopt)\in\nn\times \xin$. The optimal value is the classical value of the supremum of the reals $x^\intercal Q x+q^\intercal x$ for $x\in\rea$. 

In this paper, we will suppose that the matrix $A$ is diagonalizable and has a spectral radius strictly smaller than one. Moreover, we will perform our computations for the cases where $Q$ is positive semi-definite or negative definite making the objective function either convex or strictly concave. Even with the stability condition on $A$, an optimal solution may not exist. However, the optimal value is always finite.  

\subsection{Related works}
The author of the paper initiated a work in over-approximating the value of problem using semi-definite programming without any guarantees on the exactness of the over-approximation. The approach have been developed when the discrete-time system was piecewise affine~\cite{10.1007/978-3-319-54292-8_2} or polynomial~\cite{adje2015property}. The technique developed here avoids, when $Q$ is positive semi-definite, the use of semi-definite programming~\cite{wolkowicz2012handbook}. 

A preliminary work using Lyapunov function has also been investigated~\cite{adj2018optimal}. This current paper has practical improvements with respect to the preliminary work. The algorithm developed here provides best results since the number of iterations is smaller in practice.

The closest work seems to be the one proposed by Ahmadi and G\"unl\"uk~\cite{7403149,ahmadi2018robust}. They are interested in solving an optimization problem of the following form:
\[
\min_{x_0}\{f(x_0)\mid x_k\in \Omega,k=0,1,2,\ldots,\ x_{k+1}=g(x_k)\}
\]
where $f$ is a linear functional and $g$ is a linear function (or belongs to a finite family of some linear functions). The formulation differs from ours. First,in their context, the state variable of the system has to stay in a polyhedral invariant $\Omega$ whereas, in our context, a constraint is only imposed to $x_0$ ($x_0\in\xin$). If we rewrite our problem into their framework we should write $\Omega=\rea$. In this case, $\Omega$ is polyhedral if $\rea=\rea_k$ (the one defined at Eq.~\eqref{reachseq}) for some $k\in\nn$.  Second, their problem deals with linear objective function. Moreover, the authors propose the computation of inner and outer approximations of the reachable values set based on semi-definite programming. This is not mandatory here. Finally, their approach can be used for switched linear systems~\cite{sun2006switched}. The main similarity is the computation of an upper bound on number of iterations (the number $K$ in Equation~\eqref{motive}). Those bounds are not comparable with the one proposed here since our frameworks are different. 

Some hypotheses made in the paper (the existence of a positive term) are connected to some decision problems for discrete-time dynamical systems~\cite{fijalkow2019decidability}. Those decision problems (Skolem problem and its variants) are still open~\cite{ouaknine2014positivity}. We do not provide any result about the decidability of the existence of positive terms. First, the studied sequence is not a linear recurrence. Second, the goal is this paper is to develop a constructive method to solve computationally an optimization problem for which a positive term of a sequence must exist. 

Quadratic optimization over the trajectories of linear systems are also classical to synthesize optimal controls~\cite{880615} or perform a robust analysis or even for inverse problem in control~\cite{ZHANG2019108593}. Two main differences between problems in control theory and this paper occur. First, here, we do not consider controlled systems. The dynamical system evolves autonomously and no controls are applied. A possible link is that the control law has been synthesized before our analysis and we perform an analysis for the closed-loop system gathering the control and the state variable in one new vector. In second time, again, since the goal is to optimize along all possible orbits, this is not allowed to break the system to an arbitrary finite horizon.   

\subsection{Contributions and Outline of the paper}
In this paper, the main contribution is the resolution of the optimization problem depicted at Equation~\eqref{optpb} when the matrix $Q$ is either positive semi-definite or negative definite and the matrix $A$ is diagonalizable and has a spectral radius strictly smaller than one. The resolution of the problem means that we provide a reachable value and the optimal value of the optimization problem. A reachable value is actually characterized by a vector in $\xin$ and an integer $k\in\nn$. The key idea is the construction of a family of integers that over-approximate the smallest maximizer rank. This family is parameterized by the ranks for which the term of the sequence $\left(\max_{x\in\rea_k} x^\intercal Q x+q^\intercal x\right)_k$ is positive. The integers of the family represent a certain number of iterations to make to be sure to obtain a maximizer and the optimal value of Problem~\eqref{optpb}. A new integer of the family is computed within our algorithm, when possible, to reduce the number of iterations. The detailed method is presented at Algorithm~\ref{algoAffQP}.

Section~\ref{sequences} is devoted to an abstract study of the supremum of real sequences whose limit is zero. The goal of the study is to determine conditions (existence of a non-negative term) for the existence of a maximizer (a rank for which the associated term is the supremum of the sequence) (Proposition~\ref{posetdelta}). The study also identifies the smallest maximizer rank (Proposition~\ref{argmax}). 

Section~\ref{mainresults} applies the results of the study of Section~\ref{sequences}. For the problem presented at Equation~\eqref{optpb}, the stability of the system and the compactness of $\xin$ proves that the sequence of problem converges to zero. To solve Problem~\eqref{optpb}, we then compute an over-approximation of the smallest maximizer rank (Th~\ref{mainth}). This safe computable over-approximation of the smallest maximizer of the problem depends on a given positive term of the sequence. This over-approximation is constructed with respect to a spectral decomposition of the matrix defining the system. Section~\ref{mainresults} explains some computational aspects relative to the computation of auxiliary optimization problems and how to find the first positive term of the sequence. Section~\ref{mainresults} also presents Algorithm~\ref{algoAffQP} that permits to solve Problem~\eqref{optpb}. 

Section~\ref{experiments} is devoted to numerical examples, implementation and experiments. The examples are purely academic and illustrates in detail the potential of the techniques. Experiments are executions of Algorithm~\ref{algoAffQP} on randomly generated systems.

Section~\ref{conclusion} concludes and discusses some future direction of research.


\section{On the supremum of zero limit real sequences}
\label{sequences}
Let us denote by $\co$ the set of real sequences the limit of which is equal to 0 i.e. $\co=\{s=(u_0,u_1,\ldots)\in\rr^\nn\mid \lim_{n\to +\infty} u_n=0\}$. For an element of $\co$, we are 
interested in computing the supremum of its terms. 

For all $(n,m)\in  \nn \times (\nn\cup\{\infty\})$ such that $n<m$, we introduce the function from $\co$ to $\rr$ defined as follows:
\[
u\mapsto S^{n,m}_u=
\left\{
\begin{array}{cr}
\displaystyle{\sup_{k\in\nn} u_k} & \text{if } n=0 \text{ and } m=\infty\\
\displaystyle{\sup_{k\geq n} u_k} & \text{if } n>0 \text{ and } m=\infty\\
\displaystyle{\sup_{n\leq k\leq m} u_k} & \text{if } n>0 \text{ and } m<\infty
\end{array}
\right.
\]
Then, with our notations, for $u\in\co$, we are interested in computing $S_u^{0,\infty}$. For $u\in\co$, we are also looking for the set of maximizers i.e. the set of ranks which attain the supremum of the terms of the sequence. We will denote the set of maximizers by $\agmx(u)$. More formally, $\agmx(u)=\{k\in\nn\mid u_k=\displaystyle{S_u^{0,\infty}}\}$.
For computations purpose, we need to characterize $\agmx(u)$. Consequently, we introduce, for $u\in\co$, the two sets of ranks:
\[
\gse_u=\{k\in\nn\mid S^{0,k}_u\geq S^{k+1,\infty}_u\}\qquad \text{ and }\qquad \gs_u=\{k\in\nn\mid S^{0,k}_u> S^{k+1,\infty}_u\}
\]
It is easy to see that $\gse_u\subseteq \gs_u$. Moreover, if $k$ belongs to $\gse_u$ (resp. $\gs_u$), then any integer greater than $k$ belongs to $\gse_u$ (resp. $\gs_u$). Besides, we will need the set of ranks for which the associated term is non-negative (resp. positive):
\[
\poseq_u=\{k\in\nn\mid u_k\geq 0\}\qquad \text{ and }\qquad \pos_u=\{k\in\nn\mid u_k>0\}
\]
It is obvious that $\pos_u\subseteq \poseq_u$. We insist on the fact that even if $u\in\co$, $\poseq_u$ and hence $\pos_u$ can be empty. To construct the characterizations, we need to consider the smallest elements of $\gse_u$, $\gs_u$, $\poseq_u$ and $\pos_u$.    
\[
\bigkse_u=\inf \gse_u ;\  
\qquad 
\bigks_u=\inf\gs_u ;\
\qquad
\kse_u=\inf\poseq_u
\quad \text{ and }\ 
\ks_u=\inf\pos_u
\]
Note that by convention, the smallest element of the empty set is equal to $+\infty$. Hence, 
$\bigkse_u<+\infty$ if and only if $\gse_u\neq\emptyset$; $\bigks_u<+\infty$ if and only if $\gs_u\neq\emptyset$; $\kse_u<+\infty$ if and only if $\poseq_u\neq\emptyset$ and $\ks_u<+\infty$ if and only if $\pos_u\neq\emptyset$. Figure~\ref{fig:defks} illustrates the definition of the integers for different sequences. 

\begin{figure}[h]
\begin{minipage}[b]{0.45\textwidth}
    \begin{tikzpicture}[scale=0.9]
    \newcommand{\xint}{0,1,...,40}
    \begin{axis}[
    axis lines=left,
    axis x line =middle,
    ymax=2,
    minor tick num=5,
    x tick label style={
    above
    },
    x label style={at={(1.04,0.59)},anchor=south},
    y label style={at={(0.17,1.1)},rotate=-90,anchor=north},
    xlabel = $k$,
    ylabel = $x_k$
    ]
\foreach \x in \xint {\addplot[only marks,mark=*,mark size= 0.3mm] coordinates {(\x,{-4*abs(sin(deg((0.4*\x+0.5)*pi)))/(0.04*\x+1)})};}
    \end{axis}
        \draw (3.5,0) node {$x_k=\dfrac{-4|\sin((0.4k+0.5)\pi)|}{0.04k+1}$};
        \draw (3.5,5) node {$\kse_x=\ks_x=\bigkse_x=\bigks_x=+\infty$};
    \end{tikzpicture}
\end{minipage}
\hspace{0.5cm}
\begin{minipage}[b]{0.45\textwidth}
\begin{tikzpicture}[scale=0.9]
 \newcommand{\xint}{0,1,...,40}
    \begin{axis}[
    axis lines=left,
    axis x line =middle,
    ymax=2,
    ymin=-4,
    minor tick num=5,
    x tick label style={
    above
    },
    x label style={at={(1.04,0.59)},anchor=south},
    y label style={at={(0.17,1.1)},rotate=-90,anchor=north},
    xlabel = $k$,
    ylabel = $y_k$
    ]
\foreach \x in \xint {\addplot[only marks,mark=*,mark size= 0.3mm] coordinates
 {(\x,{-3*abs(sin(deg(0.4*(\x+1)*pi)))/(0.1*\x+1)})};}
\end{axis}
\draw (3.5,0) node {$y_k=\dfrac{-3|\sin((0.4(k+1)\pi)|}{0.1k+1}$};
\draw (3.5,5) node {$\kse_y=\bigkse_y=4;\ \ks_y=\bigks_y=+\infty$};
\end{tikzpicture}
\end{minipage}

\begin{minipage}[b]{0.45\textwidth}
\begin{tikzpicture}[scale=0.9]
 \newcommand{\xint}{0,1,...,40}
    \begin{axis}[
    axis lines=left,
    axis x line =middle,
    ymax=3,
    ymin=-4,
    minor tick num=5,
    x tick label style={
    above
    },
    x label style={at={(1,0.48)},anchor=south},
    y label style={at={(0.17,1.1)},rotate=-90,anchor=north},
    xlabel = $k$,
    ylabel = $z_k$
    ]
\foreach \x in \xint {\addplot[only marks,mark=*,mark size= 0.3mm] coordinates {(\x,{(1.6*(\x-1)/(0.08*\x*\x+0.5)})};}
\end{axis}
\draw (3.5,1) node {$z_k=\dfrac{1.6k-1.6}{0.08 k^2+0.5}$};
\draw (3.5,2.2) node {$\kse_z=1;\ \ks_z=2;\ \bigkse_z=\bigks_z=4$};
\end{tikzpicture}
\end{minipage}
\hspace{0.5cm}
\begin{minipage}[b]{0.45\textwidth}
\begin{tikzpicture}[scale=0.9]
 \newcommand{\xint}{0,1,...,40}
    \begin{axis}[
    axis lines=left,
    axis x line =middle,
    ymax=3,
    ymin=-4,
    minor tick num=5,
    x tick label style={
    above
    },
    x label style={at={(1,0.48)},anchor=south},
    y label style={at={(0.17,1.1)},rotate=-90,anchor=north},
    xlabel = $k$,
    ylabel = $t_k$
    ]
\foreach \x in \xint {\addplot[only marks,mark=*,mark size= 0.3mm] coordinates {(\x,{floor(((1.2*\x-2)/(0.04*\x*\x+0.5)))})};}
\end{axis}
\draw (3.5,1) node {$t_k=\left\lfloor \dfrac{1.2k-2}{0.04 k^2+0.5}\right\rfloor$};
\draw (3.5,2.2) node {$\kse_t=2;\ \ks_t=3;\ \bigkse_t=4;\ \bigks_t=11$};
\end{tikzpicture}
\end{minipage}
    \caption{Illustrations of the definition of the integers $\kse_u$, $\ks_u$, $\bigkse_u$ and $\bigks_u$. }
    \label{fig:defks}
\end{figure}

\begin{proposition}
\label{posetdelta}
Let $u\in\co$. The following assertions hold:
\begin{enumerate}
    \item For all $k\in\nn$, $S^{k,\infty}_u=\sup_{l\geq k} u_l\geq 0$;
    \item $\poseq_u\neq \emptyset\iff \gse_u\neq\emptyset$;
    \item $\pos_u\neq \emptyset\iff \gs_u\neq\emptyset$;
    \item $\gs_u=\emptyset\iff S_u^{0,\infty}=0$.
        \end{enumerate}
\end{proposition}
\begin{proof}
\begin{enumerate}[label = \emph{\arabic*.}]
\item Let $k\in\nn$. We are faced to two cases either there exists $l\geq k$ such that $u_l\geq 0$ or for all $l\geq k$, $u_l<0$. For the first case, we have $\sup_{l\geq k} u_l\geq 0$. For the second case, as $\lim_{k\to +\infty} u_k=0$, we have for all $\varepsilon>0$, there exists $l\geq k$, $-\varepsilon \leq u_l<0$. This means exactly that $\sup_{l\geq k} u_j=0$. Finally, in the two cases, $\sup_{l\geq k} u_l\geq 0$.

\item $\Rightarrow$. Let suppose that $\poseq_u\neq \emptyset$. This is the same as $\kse_u<+\infty$ and $\kse_u\in\poseq_u$. This can lead to two situations : either $\ks_u=+\infty$ or $\ks_u<+\infty$.
Suppose $\ks_u=+\infty$, we thus have $u_{\kse_u}=0$ and $u_j\leq 0$ for all $j\neq \kse_u$. Then $\sup_{0\leq j\leq \kse_u} u_j=u_{\kse_u}=0\geq \sup_{j>\kse_u} u_j$ and $\kse_u\in\gse_u$.
Now suppose that $\ks_u<+\infty$. Since $u_{\ks_u}>0$ and $\lim_{k\to +\infty}u_k=0$, there exists $N\in\nn$ such that $k\geq N$ implies that $u_k\leq u_{\ks_u}/2$. This implies that ${\ks_u}<N$. We thus have $\sup_{0\leq k\leq N} u_k\geq u_{\ks_u}>u_{\ks_u}/2\geq \sup_{k>N} u_k$ and $N\in\gs_u$. This proof also validates : $\pos_u\neq \emptyset\implies \gs_u\neq \emptyset$.

$\Leftarrow$.Now suppose that $\gse_u\neq \emptyset$ and let $K\in\gse_u$. Suppose that for all $k\in\nn$, $u_k<0$. Let $s=\sup_{0\leq j\leq K}u_j<0$. As $\lim_{k\to +\infty} u_k=0$, there exists $N\in\nn$ such that for all $k\geq N$, $u_k\geq s/2> s$. Then $\sup_{j>K} u_j\geq s/2>s=\sup_{0\leq l\leq K} s_l$ which contradicts the definition of $K$. 

\item As we have proved : $\pos_u\neq \emptyset\implies \gs_u\neq \emptyset$, we must prove the converse implication. Let us suppose that $\gs_u\neq \emptyset$ and take $K\in\gs_u$. From the first assertion of the proposition, we have $\sup_{0\leq l\leq K}u_l> \sup_{j> K} u_j\geq 0$. Finally, $\sup_{0\leq l\leq K}u_l >0$ which means that $\pos_u\neq \emptyset$.

\item Suppose that $\gs_u=\emptyset$. Then from the third statement, for all $k\in\nn$, we have $u_k\leq 0$. Then $S_u^{0,\infty}\leq 0$ and from the first statement $S_u^{0,\infty}\geq 0$. Finally,
$S_u^{0,\infty}=0$. 

Now, Suppose that $S_u^{0,\infty}=0$ then for all $k\in\nn$, $u_k\leq 0$ and then $\pos_u=\emptyset$ and from the third statement $\gs_u=\emptyset$.
\end{enumerate}

\end{proof}
\begin{proposition}
\label{ineg}
The following inequalities hold:
\[
\kse_u\leq \ks_u;\qquad \kse_u\leq \bigkse_u;\qquad \ks_u\leq \bigks_u\quad \text{ and }\bigkse_u\leq \bigks_u\enspace .
\]
Moreover, if $\ks_u<+\infty$, then $\ks_u\leq \bigkse_u$.
\end{proposition}
\begin{proof}
Since $\pos_u\subseteq \poseq_u$ then $\kse_u\leq \ks_u$. 
If $\bigkse_u=+\infty$, $\kse_u\leq \bigkse_u$ holds. Now, we suppose that $\bigkse_u<+\infty$. From the second statement of Prop.\ref{posetdelta}, $\kse_u<+\infty$. Suppose that $\bigkse_u<\kse_u$. Therefore, for all $k\leq\bigkse_u$, $u_k<0$ and we have $S_u^{0,\bigkse_u}<0\leq u_{\ks_u}\leq S_u^{\bigkse_u+1,\infty}$ and thus $\bigkse_u\notin \gse_u$ which contradicts its minimality and $\kse_u\leq \bigkse_u$.
The same proof can be adapted to prove $\ks_u\leq \bigks_u$. Finally, since $\gs_u\subseteq \gse_u$, we have $\bigkse_u\leq \bigks_u$. Using the same proof as for $\kse_u\leq \bigkse_u$, we can prove $\ks_u\leq \bigkse_u$ when $\ks_u<+\infty$.

\end{proof}
\begin{proposition}[Argmax]
\label{argmax}
Let $u\in\co$. The following assertions hold:
\begin{enumerate}
    \item $\agmx(u)\subseteq \gse_u$;
    \item If $\bigkse_u<+\infty$, $\bigkse_u=\min \agmx(u)$;
    \item If $\bigks_u<+\infty$, $S_u^{0,\infty}>0$ and $\bigks_u=\max \agmx(u)$;
    \item $\agmx(u)\neq\emptyset\iff \gse_u\neq\emptyset$;
    \item $\agmx(u)=\emptyset \implies S^{0,\infty}_u=\lim_{k\to +\infty} u_k=0$; 
\end{enumerate}
\end{proposition}
\begin{proof}
\begin{enumerate}[label = \emph{\arabic*.}]
\item Let $k\in\agmx(u)$. We have $u_k=S^{0,\infty}_u$ then $S_u^{0,k}=u_k\geq S_{u}^{k+1,\infty}$ and $k\in\gse_u$.

\item Assume that $\bigkse_u<+\infty$. From the first statement, $k\in\agmx(u)$ implies that $\bigkse_u\leq k$. It suffices to prove that $\bigkse_u\in\agmx(u)$. However, we have : $\sup_{l\in\nn} u_l=\max\{\sup_{0\leq j\leq \bigkse_u} u_j,\sup_{m>\bigkse_u}u_m\}=\sup_{0\leq j\leq \bigkse_u} u_j$. Therefore, there exists $j\in\agmx(u)$ such that $j\leq \bigkse_u$. This integer $j$ must also satisfy $\bigkse_u\leq j$. We conclude that $j=\bigkse_u$ and $\bigkse_u\in\agmx(u)$.

\item Suppose that $\bigks_u<+\infty$. Suppose that there exists $k\in\agmx(u)$ such that $\bigks_u<k$.
From the definition of $\bigks_u$, we have $S_u^{0,\infty}\geq S_u^{0,\bigks_u}>S_u^{\bigks_u+1,\infty}\geq u_k=S_u^{0,\infty}$ which is not possible. Then $k\leq \bigks_u$. Moreover, $S_u^{0,\infty}\geq S_{u}^{0,\bigks_u}>S_u^{\bigks_u+1,\infty}\geq 0$ from Prop.~\ref{posetdelta}. Now it suffices to prove that $u_{\bigks_u}=\sup_{0\leq l\leq \bigks_u} u_l$. Suppose that $u_{\bigks_u}$ does not attain the maximum and let $\overline l=\max\{0\leq l\leq \bigks_u\mid u_l=S_u^{0,\bigks_u}\}$. From the definition of $\bigks_u$, $\sup_{l>\bigks_u} u_l<\sup_{0\leq l\leq \bigks_u} u_l=u_{\overline l}$. Now, $u_{\overline l}>\sup_{\overline l<l\leq \bigks_u} u_l$ by definition of $\overline l$. 
Finally, $u_{\overline l}=\sup_{0\leq l\leq \overline l} u_l>\sup_{l>\overline l} u_l$ and $\overline l<\bigks_u$. This contradicts the minimality of $\bigks_u$ and $u_{\bigks_u}=\sup_{0\leq l\leq \bigks_u} u_l$.

\item The implication $\agmx(u)\neq \emptyset \implies \gse_u\neq \emptyset$ follows readily from the first statement. If $\gse_u\neq \emptyset$ then $\bigkse_u<+\infty$ and $\bigkse_u\in\agmx(u)\neq \emptyset$.

\item Assume that $\agmx(u)=\emptyset$. This is equivalent to $\gse_u=\emptyset$ which implies that $\gs_u=\emptyset$ and we conclude from the fourth statement of Prop.~\ref{posetdelta}.
\end{enumerate}

\end{proof}

Proposition~\ref{argmax} confirms the illustrations depicted at Figure~\ref{fig:defks}. For the sequence $(z_k)_k$, of Figure~\ref{fig:defks}, $\bigkse_u$ and $\bigks_u$ coincide. In this case, the maximizer is unique. We observe,  still for the sequence $(z_k)_k$, of Figure~\ref{fig:defks}, that the maximizer seems to satisfy a first-order condition. For the sequence $(t_k)_k$, of Figure~\ref{fig:defks}, $\bigkse_u$ is strictly smaller than $\bigks_u$. Between, those two integers, the sequence is constant and, in this interval, the terms are equal to the maximum value reached by the sequence.

\begin{example}[Illustration of the fifth statement of Prop.~\ref{argmax}]
\label{counterex}
Let us consider the optimization problem in dimension one, with data:
\[
\xin=[1/4;1/2];\ x_{k+1}=(1/2)x_k;\ Q=1\text{ and } q=-1
\]
The optimization problem to solve is thus :
\[
\sup_{k\in\nn} \sup_{x\in [1/4;1/2]} ((1/2)^k x)^2 - (1/2)^k x
\]
The functions $f_k:x\mapsto ((1/2)^k x)^2-(1/2)^k x$ are strictly decreasing on $[1/4;1/2]$ then $u_k:=\sup_{x\in [1/4;1/2]} f_k(x)=(1/16)\times (1/2)^{2k}-(1/4)\times (1/2)^k$.

We have for all $k\in\nn$, $u_k<0$ and thus $\poseq_u=\emptyset$. The sequence $(u_k)_k$ is strictly increasing and thus $\gse_u=\emptyset$. The sequence $(u_k)_k$ tends to 0. In this example, we have $\sup_{k\in\nn}u_k=\lim_{k\to +\infty} u_k=0$. The supremum cannot be computed in finite time.

\end{example}

\begin{proposition}
\label{supremumpos}
Let $u\in\co$. The following assertions hold:
\begin{enumerate}
    \item Assume that $\poseq_u\neq \emptyset$. Let $k\geq \kse_u$ such that $S_u^{\kse_u,k}\geq S_u^{k+1,\infty}$ then $k\in\gse_u$;
    \item Assume that $\pos_u\neq \emptyset$. Let $k\geq \ks_u$ such that $S_u^{\ks_u,k}\geq S_u^{k+1,\infty}$ then $k\in\gse_u$;
    \item Assume that $\pos_u\neq \emptyset$. For all $\bigkse_u\leq k$, $S_u^{0,k}=S_u^{\ks_u,k}=S_u^{\kse_u,k}=S_u^{0,\infty}$;
\end{enumerate}
\end{proposition}
\begin{proof}
\begin{enumerate}[label = \emph{\arabic*.}]
\item Let $k\geq \kse_u$. Let $0\leq j\leq \kse_u$. Then $u_j<0\leq u_{\kse_u}\leq S_u^{\kse_u,k}$. Hence, $S_u^{0,k}=S_u^{\kse_u,k}$ and since $S_u^{\kse_u,k}\geq S_u^{k+1,\infty}$, we conclude that $k\in\gse_u$.

\item The same proof as for the first point can be applied. 

\item Let $k\geq \bigkse_u$. From Prop.~\ref{ineg}, $\kse_u$ and $\ks_u$ are smaller than $\bigkse_u$ and $\kse_u\leq \ks_u$. Therefore, $S_u^{0,\infty}\geq S_u^{\kse_u,k}\geq u_{\bigkse_u}$ and $S_u^{0,\infty}\geq S_u^{\ks_u,k}\geq u_{\bigkse_u}$. 
From Prop.~\ref{argmax}, $u_{\bigkse_u}=S_u^{0,\infty}$ and the result holds.
\end{enumerate}

\end{proof}  
In summary, for $u\in\co$, to compute $S_u^{0,\infty}=\sup_{k\in\nn} u_k$, we need to study first the emptiness of $\pos_u$. Indeed, if $\pos_u$ is empty, we know that (fourth statement of Prop.~\ref{posetdelta}) $S_u^{0,\infty}$ is equal to 0. If $\pos_u$ is not empty, then we have to compute $\bigkse_u$ to find $S_u^{0,\infty}$ and a maximizer (second statement of Prop.~\ref{argmax}). However, to identify $\bigkse_u$, we need knowledge on the past and the future of the sequence. The good point is that any over-approximation $k$ of $\bigkse_u$ permits to know $S_u^{0,\infty}$ by computing $S_u^{0,k}$ (fourth statement of Prop~\ref{supremumpos}).
\section{Maximization of a quadratic form over the reachable values set}
\label{mainresults}
We come back to Problem~\eqref{optpb}. In this section, we suppose that 
$b=0$. We will describe later how to deal with the case $b\neq 0$ in Subsection~\ref{affine}. Moreover, if $b=0$, since $\xin=\{0\}$ implies trivially $S_\nu^{0,\infty}=0$, we assume that $\xin\neq \{0\}$. 

We introduce the sequence $\nu$ defined for all $k\in\nn$ by: 
\begin{equation}
\label{eqnuk}
\nu_k=\sup_{x\in\xin} x^\intercal {A^k}^\intercal Q A^k x+q^\intercal A^k x
\end{equation}
Problem~\eqref{optpb} in the case where $b=0$ is equivalent to compute $\sup_{k\in\nn} \nu_k$. We make the following assumption:
\begin{assumption}
\label{assum1}
The spectral radius of $A$, $\rho(A)$, satisfy $\rho(A)<1$. 
\end{assumption}

Recall that $\xin$ is a polytope and then is bounded, the following proposition thus holds.
\begin{proposition}
\label{propnuk0}
Assumption~\ref{assum1} implies that 
$\displaystyle{\lim_{k\to +\infty} \nu_k=0}$ i.e. $\nu\in\co$.
\end{proposition}
\comment{
\begin{proof}
Actually $A^k$ tends to the zero matrix for any matrix norm. Then using standard arguments from quadratic forms theory and Cauchy-Schwartz inequality, we have for all $x\in\rd$, 
$|x^\intercal {A^k}^\intercal Q A^k x+q^\intercal A^k x|\leq \norm{Q}_2^2\norm{A^k x}_2^2+\norm{q}_2 \norm{A^k x}_2\leq \norm{Q}_2^2\norm{A^k}_2^2 \norm{x}_2^2+\norm{q}_2 \norm{A^k}_2\norm{x}_2
$.
Now since $\norm{\cdot}_2$ is continuous then it is bounded by $M$ on the compact set $\xin$ and we conclude that $|\nu_k|\leq \norm{Q}_2^2 M^2 \norm{A^k}_2^2+\norm{q}_2 M \norm{A^k}_2$. As $\norm{A^k}_2$ tends to zero as $k$ tends to $+\infty$ we get  $\lim_{k\to +\infty} \nu_k=0$.
\end{proof}
}
Prop.~\ref{propnuk0} allows to use the results of Section~\ref{sequences} and the integers $\ks_\nu$, $\kse_\nu$, $\bigks_\nu$ and $\bigkse_\nu$ relative to the sequence $(\nu_k)_{k\in\nn}$ defined at Eq.~\eqref{eqnuk}.

\begin{assumption}
\label{diagonal}
The matrix $A$ has spectral decomposition i.e. there exists a non-singular complex matrix $U$ and a complex diagonal matrix $D$ such that: 
\begin{equation}
\label{spectral}
    A=U D U^{-1}
\end{equation}
\end{assumption}

\begin{assumption}
\label{maxeigq}
The greatest eigenvalue of $Q$ is not null i.e. $\lmax{Q}\neq 0$.
\end{assumption}
As $Q$ is symmetric (real Hermitian) and $U$ is non-singular, we can use Ostrowski's theorem~\cite[Th. 4.5.9]{horn1990matrix} and we get the following lemma.
\begin{lemma}
Assumption~\ref{maxeigq} is equivalent to $\lmax{U^* Q U}\neq 0$.
\end{lemma}
\comment{
\begin{proof}
Since $U^* Q U$ is Hermitian, then $\lmax{U^* Q U}=0$ is equivalent to the fact that for all $x\in\cc^d$, $x^*U^* Q U x\leq 0$ and $y^*U^* Q U y=0$ for a non-zero $y\in\cc^d$. As $U$ is non-singular then $\lmax{U^* Q U}=0$ is equivalent to the fact that for all $z\in\cc^d$, $z^*Q z\leq 0$ and $u^*Q u=0$ for a non-zero $u\in\cc^d$.
\end{proof}
}
In this section, first, we construct an over-approximation of $\bigkse_\nu$. We are looking for the smallest over-approximation possible. Actually, following the fourth statement of Prop.~\ref{supremumpos}, the number of evaluations depends on the quality of this over-approximation. In a second time, we explain the computation of a term $\nu_k$ as Eq~\eqref{eqnuk} indicates the resolution of a constrained quadratic maximization problem. Then, we discuss the existence and the computation of $\ks_\nu$. Finally, we end the section with the main result.
\subsection{Computing an over-approximation of $\bigkse_\nu$ using the spectral decomposition}
\label{bigksenu}
Recall that for a Hermitian matrix $B$, for all $x\in\cd$, $x^* B x$ is real scalar. Hence, we define by, for a Hermitian matrix $B$:
\begin{equation}
\label{mudef}
\mus{B}=\sup_{x\in\xin} x^\intercal B x
\end{equation}

Let us introduce 
\begin{equation}
    \label{auxdiag}
    \vdiag:=\dfrac{\norm{U^*q}_2}{2\sqrt{|\lmax{U^*QU}|}}
\end{equation}

\begin{proposition}
\label{nukineg}
For all $k>0$, we have $\nu_k\leq \left(\sqrt{|\lmax{U^*QU}|\mus{(UU^*)^{-1}}}+\vdiag\right)^2-\vdiag^2$.
\end{proposition}
\begin{proof} 
Let $x\in\xin$. As $x\in\rd$ and $A$ is a $d\times d$ real matrix, $x^\intercal A^\intercal =x^* A^*$.
We denote by $U^{-*}$ the conjugate transpose of the inverse of $U$ and $|D|^2=D^* D$ the diagonal matrix composed of the square of modulus of the eigenvalues of $A$ on its diagonal.
\[
\begin{array}{ll}
x^*{A^k}^* Q A^k x=x^* U^{-*}{D^*}^k U^* Q U D^k U^{-1} x&\leq |\lmax{U^* Q U}|x^* U^{-*}|D|^{2k} U^{-1} x\\
                                 &\leq \rho(A)^{2k}|\lmax{U^* QU}| \norm{U^{-1} x}_2^2\\
                                 &\leq \rho(A)^{2k}|\lmax{U^* QU}| \mus{(UU^*)^{-1}}
                                 \end{array}
\]
Moreover, using Cauchy-Schwarz in $\cd$ and $\norm{U^{-1}x}_2=\sqrt{x^* U^{-*} U^{-1}x}$, we have:
\[
q^\intercal A^k x=q^* A^k x=q^* U D^k U^{-1} x\leq \rho(A)^k\norm{U^* q}_2\norm{U^{-1}x} _2
=\rho(A)^k\norm{U^* q}_2\sqrt{\mus{(UU^*)^{-1}}}
\]

By summing the two terms, we get:

\begin{equation}
\label{eqfond}
\begin{array}{cl}
\nu_k&\leq \displaystyle{\rho(A)^{2k}|\lmax{U^* QU}| \mu ((UU^*)^{-1})+\rho(A)^k\norm{U^* q}_2\sqrt{\mus {(UU^*)^{-1}}}}\\
&=\displaystyle{\left(\rho(A)^k\sqrt{|\lmax{U^*QU}|\mus{(UU^*)^{-1}}}+\vdiag\right)^2-\vdiag^2}
\end{array}
\end{equation}
As $(\rho(A)^k\sqrt{|\lmax{U^*QU}|\mus{(UU^*)^{-1}}}+\vdiag\geq 0$ and $\rho(A)<1$ the inequality holds.

\end{proof}

\begin{corollary}
\label{corollary1}
If $\nu_0\geq \left(\sqrt{|\lmax{U^*QU}|\mus{(UU^*)^{-1}}}+\vdiag\right)^2-\vdiag^2$ then 
    $\nu_0=S_\nu^{0,\infty}=\displaystyle{\sup_{k\in\nn} \nu_k}$. 
\end{corollary}

\begin{theorem}
\label{mainth}
Let $j\in\pos_\nu$. We define the integer :
\begin{equation}
\label{kdiagform}
\kdiag_\nu(j):=\left\lfloor\dfrac{\ln\left(\left(\sqrt{\nu_j+\vdiag^2}-\vdiag\right)\left(\sqrt{|\lmax{U^*QU}|\mus{(UU^{*})^{-1}}}\right)^{-1}\right)}{\ln(\rho(A))}\right\rfloor +1
\end{equation}
Then:
\begin{enumerate}
\item $\kdiag_\nu(j)$ is well-defined i.e an integer greater than 1;
\item For all $k\geq \kdiag_\nu(j)$, $\nu_k\leq \nu_j$.
\end{enumerate}
\end{theorem}
\begin{proof}
\begin{enumerate}[label = \emph{\arabic*.}]
\item Let $j\in\pos_\nu$. As $\nu_j>0$, we have $\sqrt{\nu_j+\vdiag^2}-\vdiag>0$. From Prop.~\ref{nukineg}, we have $\sqrt{\nu_j+\vdiag^2}-\vdiag\leq \sqrt{|\lmax{U^*QU}|\mus{(UU^{*})^{-1}}}$. Moreover, by assumption, $\rho(A)<1$. Therefore, the denominator in Eq.~\eqref{kdiagform} is negative and the numerator is non-positive. In consequence $\kdiag_\nu(j)\geq 1$.

\item Let $j\in\pos_\nu$
As $\rho(A)<1$, we have, for all $k\geq \kdiag_\nu(j)$, $\rho(A)^k\leq \rho(A)^{\kdiag_\nu(j)}$. Thus using the natural logarithm:
\[
\ln(\rho(A)^k)\leq \kdiag_\nu(j)\ln(\rho(A))\leq \ln(\rho(A)) \dfrac{\ln\left(\left(\sqrt{\nu_j+\vdiag^2}-\vdiag\right)\left(\sqrt{|\lmax{U^*QU}|\mus{(UU^{*})^{-1}}}\right)^{-1}\right)}{\ln(\rho(A))}\enspace .
\]
So, 
$\rho(A)^k\leq \left(\sqrt{\nu_j+\vdiag^2}-\vdiag\right)\left(\sqrt{|\lmax{U^*QU}|\mus{(UU^{*})^{-1}}}\right)^{-1}$ and
$(\rho(A)^k \sqrt{|\lmax{U^*QU}|\mus{(UU^{*})^{-1}}}+\vdiag)^2-\vdiag^2\leq \nu_j$.
From Eq.~\eqref{eqfond}, we conclude that $\nu_k\leq \nu_j$. 
\end{enumerate}

\end{proof}

\begin{proposition}
\label{inter}
The following statements hold:
\begin{enumerate}
\item Let $j,k\in\pos_\nu$. If $\nu_k\leq \nu_j$ then $\kdiag_\nu(j)\leq \kdiag_\nu(k)$.
\item Let $j\in\pos_\nu$. We have, if $j\notin\agmx(\nu)$, $\bigkse_\nu\leq \kdiag_\nu(j)$; and $\bigkse_\nu\leq \max\{j,\kdiag_\nu(j)\}$ otherwise.
\end{enumerate}
\end{proposition}

\begin{proof}
\begin{enumerate}[label = \emph{\arabic*.}]
\item The result is a direct consequence of two arguments: $\ln(\rho(A))<0$ and the function $x\mapsto \ln((\sqrt{x+\vdiag^2}-\vdiag)/\sqrt{|\lmax{U^*QU}|\mus{(UU^{*})^{-1}}})$ is increasing.

\item Let $j\in\pos_\nu$. Suppose that $j\notin\agmx(\nu)$. If $\kdiag_\nu(j)<\bigkse_\nu$, then as for all $k\geq\kdiag_\nu(j)$, $\nu_k\leq \nu_j$, we have $\nu_{\bigkse_\nu}=S_\nu^{0,\infty}\leq \nu_j<S_\nu^{0,\infty}$ which is not possible and $\bigkse_\nu\leq \kdiag_\nu(j)$.
 Now if $j\in\agmx(\nu)$, we have from Prop.~\ref{argmax}, $\bigkse_\nu\leq j$ and then $\bigkse_\nu\leq \max\{j,\kdiag_\nu(j)\}$.
 \end{enumerate}
 
 \end{proof}



\begin{corollary}
\label{final}
Let $\displaystyle{\mink_\nu=\min_{k\in \pos_\nu} \max\{k,\kdiag_\nu(k)\}}$.
The following statements hold:
\begin{enumerate}
    \item $\mink_\nu=\max\{\bigkse_\nu,\kdiag_\nu(\bigkse_\nu)\}$;
    \item For all $k\in\pos$, $S_\nu^{0,\infty} =S_\nu^{0,\mink_\nu}=S_\nu^{0,\max\{k,\kdiag_\nu(k)\}}$.
\end{enumerate}
\end{corollary}
\begin{proof}
\begin{enumerate}[label=\emph{\arabic*.}]
\item Let $k\in\pos_\nu$. We have proved that if $k\notin \agmx(\nu)$ then $\bigkse_\nu\leq \kdiag_\nu(k)$ and if $k\in\agmx(\nu)$ then $\bigkse_\nu\leq\max\{k,\kdiag_\nu(k)\}$. In both cases, 
 $\bigkse_\nu\leq\max\{k,\kdiag_\nu(k)\}$. 
Now, as $\nu_k\leq \nu_{\bigkse_\nu}$, 
we get, from Prop.~\ref{inter}, $\kdiag_\nu(\bigkse_\nu)\leq \kdiag_\nu(k)$. Finally, $\max\{\bigkse_\nu,\kdiag_\nu(\bigkse_\nu)\}\leq \max\{\max\{k,\kdiag_\nu(k)\},\kdiag_\nu(k)\}=
\max\{k,\kdiag_\nu(k)\}$.

\item The results follows readily from the third statement of Prop.~\ref{supremumpos} as $\mink_\nu$ and 
for all $k\in\pos_\nu$, $\max\{k,\kdiag_\nu(k)\}$ are greater than $\bigkse_\nu$. 
\end{enumerate}

\end{proof}

The second statement of Corollary~\ref{final} means that we can find the optimal value of Problem~\eqref{optpb} by solving exactly $\kdiag_\nu(k)$ maximization problems where $k\in\pos_\nu$. The integer using the form~\eqref{kdiagform} which needs less computations is then $\max\{\bigkse_\nu,\kdiag_\nu(\bigkse_\nu)\}$. However, we cannot decide whether $k=\bigkse_\nu$ and thus we will use an arbitrary element of $\pos_\nu$. The difficulties is now to compute $\nu_j$ for all $j=0,1,\ldots,\max\{k,\kdiag_\nu(k)\}$, to decide whether $\pos_\nu$ is non-empty and to find an integer within $\pos_\nu$. 

\subsection{About the computation of $\nu_k$}
\label{computingnuk}
Even if we have reduced the initial infinite optimization problem to a finite number of computations, we have to compute $\nu_k$ for all $k=0,\ldots,\kdiag_\nu(j)$. The computational aspects rely on the nature of the matrix $Q$ and the vector $q$. We introduce the functions: 
\begin{equation}
\label{fk}
f_k:x\mapsto x^\intercal {A^k}^\intercal Q A^k x+q^\intercal A^k x.
\end{equation}
With this new notation, $\nu_k=\sup_{x\in\xin} f_k(x)$ for all $k\in\nn$.
\paragraph{$Q$ indefinite case}
If $Q$ is indefinite, the functions $f_k$ are, in general, neither convex nor concave. Hence, for all $k\in\nn$, $\nu_k$ is the optimal value of an indefinite quadratic program. It is well-known (e.g see the survey in the paper~\cite{furini2019qplib}) that the problem is NP-hard and the current algorithms and solvers can only guarantee to find a local maximizer~\cite{absil2007newton,huyer2018minq8}. We could consider specific situations where the global maximizer can be computed~\cite{zhang2000quadratic,kim2003exact,burer2019exact}. In consequence, we do not treat this class in the paper.
 
\paragraph{Strictly concave objective function}
We suppose that $Q$ is definite negative making the functions $f_k$ concave, for all $k\in\nn$. First, we warn the reader that the case where $Q$ is only semi-definite negative is not compatible with Assumption~\ref{maxeigq}. Second, if $q$ is equal to 0, the function $f_k(x)$ are non-positive for all $x\in\rd$ and $\pos_\nu=\emptyset$. Thus $S_\nu^{0,\infty}=0$ (Prop.~\ref{posetdelta}). Now, if $q$ is different to 0, we can use any convex quadratic programming solver. Those solvers can be based on general non-linear methods, for example, interior points methods~\cite{vanderbei1999loqo,friedlander2012primal}, non-interior points methods~\cite{tian2015exterior} or active-sets methods~\cite{curtis2015globally,forsgren2016primal}.

\begin{theorem}
Let $k$ be a given element of $\pos_\nu$. If $Q\preceq 0$ and $q\neq 0$, then Problem~\eqref{optpb} can be solved using $\max\{k,\kdiag_\nu(k)\}$ convex quadratic problems. 
\end{theorem}
\paragraph{Convex objective function}
When $Q$ is semi-definite positive implying that the functions $f_k$ are convex. We can use Lemma~\ref{lemma2} to compute $\nu_k$ for all $k\in\nn$, with a finite number of evaluations. The number of evaluations is exactly the number of vertices of $\xin$. 

\begin{lemma}[Maximization of a convex function over a polytope]
\label{lemma2}
Let $C$ be a polytope and $f:\rd\to\rr$ be a convex function. Let us denote by $\mathcal{E}(C)$ the finite set of vertices of $C$. Then: $\max_{x\in C} f(x)=\max_{x\in \mathcal{E}(C)} f(x)$.
\end{lemma}

\begin{theorem}
Suppose that an element $k$ in $\pos_\nu$ is given. If $Q\succeq 0$, then Problem~\eqref{optpb} can be solved using $\max\{k,\kdiag_\nu(k)\}\times \mathcal{E}(\xin)$ evaluations of quadratic expressions. 
\end{theorem}

Unfortunately, to use this method to solve a concave quadratic program does not scale well. The resolution time and the memory consumption  blow up when the number of vertices grows. In future work, we will call more tractable methods to solve high dimensional problems.  
\subsection{About $k\in\pos_\nu$}
\label{posnu}
The computation of $\kdiag_\nu$ needs an element of $\pos_\nu$. So, we have to answer to two questions: is $\pos_\nu$ empty? How to compute $\ks_\nu$ i.e. the smallest element of $\pos_\nu$?
 
To decide whether $\pos_\nu\neq \emptyset$ is a difficult problem. This is the same to ask whether $-\nu_k>0$ for all $k\in\nn$. For a linear recurrence, this problem is called the positivity problem~\cite{ouaknine2014positivity}. Note that for linear recurrences, this problem is still open. In consequence, in practice, we fix a maximal number of visited ranks denoted by $N$. If we have not found a positive iterate $\nu_k$ before $N$, then we abort the computation and return the status "failed". 

We also address another simpler question: $0\in\pos_\nu$? This is the only case which does not involve the powers of the matrix $A$. So, those simple situations take into account $\xin$, $Q$ and $q$. First, given $\xin$, $Q$ and $q$, we can check whether $0\in\pos_\nu$ using the same arguments as in Subsection~\ref{computingnuk}. Second, we want to identify simple situations ensuring that $0\in\pos_\nu$.

\begin{proposition}
For all $k\geq 0$, the decision problem $\nu_k>0$ can be decided using convex quadratic programming when $Q\preceq 0$ and in $\mathcal{E}(\xin)$ evaluations of quadratic expressions when $Q\succeq 0$. 
\end{proposition}

\begin{proposition}
\label{naiveresult}
If one of the following statements hold:
\begin{itemize}
    \item If $q=0$ and $Q\succ 0$;
    \item If $q=0$, $Q\succeq 0$, $det(Q)=0$ and $\inte(\xin)\neq \emptyset$;
    \item If $q\neq 0$, $Q\succeq 0$ and $0\in \inte(\xin)$.
\end{itemize}

Then $\ks_\nu=0$.
\end{proposition}
\begin{proof}
First assume that $q=0$ and $Q\succ 0$. It follows that for all non-zero $x^\intercal Q x>0$ and $\nu_0>0$ if $\xin$ is not reduced to the singleton $\{0\}$ which has been supposed earlier. 

Now assume that $q=0$, $Q\succeq 0$, $det(Q)=0$ and $\inte(\xin)\neq \emptyset$. Therefore there exists $y\in\rd$ such that $y^\intercal Q y>0$. Now suppose that $x^\intercal Q x$ is null on $\xin$. Let $z$ in the interior of $\xin$, then there exists $\epsilon>0$ such that $z\pm \epsilon y/\norm{y}_{\infty}$ belong to $\xin$. Then
$(z+\epsilon y/\norm{y}_{\infty})^\intercal Q (z+\epsilon y/\norm{y}_{\infty})=z^\intercal Q z+(\epsilon/\norm{y}_{\infty})^2 y^\intercal Q y\pm 2(\epsilon/\norm{y}_{\infty}) z^\intercal Q y=
(\epsilon/\norm{y}_{\infty})^2 y^\intercal Q y\pm 2(\epsilon/\norm{y}_{\infty}) z^\intercal Q y=
0$ by assumption. This leads to $y^\intercal Q y=\pm 2(\norm{y}_{\infty}/\epsilon) z^\intercal Q y$ which contradicts the strict positivity of $y^\intercal Q y$.

Finally, suppose that $q\neq 0$, $Q\succeq 0$ and $0\in \inte(\xin)$. Thus there exists $\gamma>0$ such that $[-\gamma,\gamma]^d\subset \xin$. Let $x=\gamma q
\norm{q}_{\infty}^{-1}$. Then $q^\intercal x=\gamma\norm{q}_2^2\norm{q}_{\infty}^{-1} >0$ and $x\in [-\gamma,\gamma]^d$. From $Q\succeq 0$, $x^\intercal Q x+q^\intercal x>0$. We thus have $\sup_{y\in \xin}  y^\intercal Q y+q^\intercal y>0$ and $\ks=0$.

\end{proof}
We end the section with the following theorem.

\begin{theorem}
Suppose that Assumptions~\ref{assum1} and~\ref{diagonal} hold. We also assume that $Q\succeq 0$ and $Q\neq 0$. Moreover, if one these statements hold:
\begin{enumerate}
\item $q=0$ and $Q\succ 0$; 
\item $q=0$ and $\inte{\xin}\neq \emptyset$;
\item $q\neq 0$ and $0\in\inte{\xin}$
\end{enumerate}

Then Problem~\eqref{optpb} can be computed in finite time i.e.
\[
\sup_{k\in\nn}\nu_k=\max_{0\leq k\leq \kdiag_\nu(0)} \max_{x\in\mathcal{E}(\xin)} x^\intercal {A^k}^\intercal Q A^k x+q^\intercal A^k x
\]
\end{theorem}
\subsection{Algorithm to solve Problem~\eqref{optpb}}
\label{algo}
We provide here an algorithm to solve Problem~\eqref{optpb} when the objective function is strictly concave or convex i.e. when $Q\prec 0$ or $Q\succeq 0$. Recall that Assumption~\ref{maxeigq} forces $Q$ to be negative definite. 

If $\xin=\{0\}$ or ($Q\prec 0$ and $q=0$) then we know that $S_\nu^{0,\infty}=0$. In consequence, in addition to Assumptions~\ref{assum1}--\ref{maxeigq}, we make the following assumptions on the inputs of the algorithm:
\begin{itemize}
\item If $Q\prec 0$ then $q\neq 0$;
\item $\xin\neq \{0\}$.
\end{itemize}

The only difference between the treatment of the strictly concave and the convex case is the resolution of each $\sup_{x\in\xin} f_k(x)$. The resolution of this optimization problem appears as the call of the oracle SolverQP. Its inputs are the objective function and the polytope defining the constraints (here $\xin$). Its outputs are the optimal value and a maximizer. When $Q\prec 0$, SolverQP is just a solver for convex quadratic programming. When $Q\succeq 0$ we use Lemma~\ref{lemma2} and explore all vertices of $\xin$. Then in this case, the maximizer returned is a vertex of $\xin$. In future work, we should consider more scalable approaches such as~\cite{konno1976maximization,Floudas1995,Tuy2016}. 
 



\begin{algorithm}[h!]
\DontPrintSemicolon

\Input{The objective function defined from $Q$ and $q$, the system defined from $A$ and $\xin$ and an integer $N$ to stop the search of a positive term.} 
\Output{A vector $(\nuopt,\xopt,\kopt)$ where $\nuopt=S_\nu^{0,\infty}$, $f_{\kopt}(\xopt)=S_\nu^{0,\infty}$ and $(\xopt,\kopt)\in\xin\times \nn$ or a status "failed" if $\nu_k\leq 0$ for all $k=0,\ldots,N$.}
\Begin{
Compute $\mu((UU^*)^{-1})$ and $\vdiag$\;
$(\nu_0,x_0)$=SolveQP($f_0$,$\xin$)\;
\If{$\nu_0\geq \left(\sqrt{|\lmax{U^*QU}|\mus{(UU^*)^{-1}}}+\vdiag\right)^2-\vdiag^2$}
{Return $(\nu_0,x_0,0)$} 
\Else
{
$k=0$\;
\While{$k<N$ and $\nu_k\leq 0$}
{
$k=k+1$\;
$(\nu_k,x_k)$=SolveQP($f_k$,$\xin$)\;
}

\If{$k=N$ and $\nu_k\leq 0$}{
Return "failed"
}
\Else{
$K=\kdiag_\nu(k)$\;
$\nuopt=\nu_k$; $\xopt=x_k$; $\kopt=k$\;
\While{$k<K$}{
$k=k+1$\;
$(\nu_k,x_k)$=SolveQP($f_k$,$\xin$)\;
\If{$\nuopt<\nu_k$}{
$\nuopt=\nu_k; \xopt=x_k; \kopt=k$\;
$K=\kdiag_\nu(k)$\;
}
}
Return $(\nuopt,\xopt,\kopt)$\;
}
}
}
\caption{Resolution of Pb.\eqref{optpb} for a convex or a strictly concave quadratic objective function}
\label{algoAffQP}
\end{algorithm}

\begin{proposition}
\label{kdiagit}
If $\ks_\nu\leq N$, then the sequence generated at Line 21 of Algorithm~\ref{algoAffQP} $\left(\kdiag_\nu(k)\right)_{k\in \Gamma}$ is decreasing where $\Gamma$ denotes the set of integers where $\nu_k>\nu_j$ for all $j<k$.
\end{proposition}

\begin{proof}
If $\ks_\nu\leq N$ and the value $\nuopt$ is modified at the step $k$, then $k\in\pos_\nu$. Let $j,k\in \Gamma$ such that $j<k$. By definition of $\Gamma$, $\nu_j<\nu_k$. From Prop.~\ref{inter}, we have $\kdiag_\nu(k)\leq \kdiag_\nu(j)$. As $\bigkse_\nu$ is the greatest element of $\Gamma$, the value $\kdiag_\nu(k)$ is minimal when $k=\bigkse_\nu$. When this value is reached $\kdiag_\nu(k)$ cannot be updated as $\nuopt$ cannot be modified.
\end{proof}

Prop.~\ref{kdiagit} proves that we reduce the number of iterations by recomputing $\kdiag_\nu(k)$ at Line 21 of Algorithm~\ref{algoAffQP}.

%

\begin{theorem}[Algorithm~\ref{algoAffQP} is correct]
Let $N\in\nn$ be fixed. Suppose that $\ks_\nu\leq N$ then Algorithm~\ref{algoAffQP} returns the optimal value, a couple of maximizers $(\kopt,\xopt)\in\nn\times \xin$ for Problem~\eqref{optpb}. 
\end{theorem}

\begin{proof}
If $\ks_\nu\leq N$, then Algorithm~\ref{algoAffQP} cannot return a failed status. Then, Algorithm~\ref{algoAffQP} stops either at Line 5 or at Line 16. If the stop is at Line 5, it means that $\nu_0\geq \left(\sqrt{|\lmax{U^*QU}|\mus{(UU^*)^{-1}}}+\vdiag\right)^2-\vdiag^2$ then the conclusion follows from Corollary~\ref{corollary1}. 
In the other case, the loop iteration starting at Line 16 terminates as $K\leq k$. Either, $\kdiag_\nu(k)\leq k$ or $\kdiag_\nu(j)=k$ for some $j<k$. If $\kdiag_\nu(k)\leq k$, from the second statement of Prop.~\ref{inter}, we get $\bigkse_\nu \leq k$. 
If $\kdiag_\nu(j)=k$ for some $j<k$. We must have $j\leq \kdiag_\nu(j)$. If $j\in\agmx(\nu)$, we have $\bigkse_\nu\leq j\leq \kdiag_\nu(j)=k$. If $j\notin\agmx(\nu)$, we have $\bigkse_\nu\leq \kdiag_\nu(j)=k$. In all situations, we have $\bigkse_\nu\leq k$ and thus the optimal value has been found.
 

\end{proof}

\subsection{From linear systems to affine ones}
\label{affine}
We come back to the recurrence formulation of Equation~\eqref{stateeq}: the case where the system is purely affine ($b\neq 0$). We adopt the basic approach which consists in using an auxiliary linear discrete-time system. We are interested in solving Problem~\eqref{optpb} where $b\neq 0$. Assumption~\ref{assum1} still holds and it implies that $\Idd-A$ is invertible. It is well-known that:
\[
\forall\, k\in\nn,\ y_k=x_k-\btilde,\text{ where } \btilde=(\Idd-A)^{-1} b \implies
\forall\, k\in\nn,\ y_{k+1}=A y_k \text{ and }
x_k=A^k y_0+\btilde
\]
This latter expression leads to a new formulation of Problem~\eqref{optpb}: 
\begin{equation}
\label{affineequation}
    \sup_{k\in\nn} \sup_{y_0\in\xin-\btilde} y_0^\intercal (A^k)^\intercal Q  A^k y_0+(2\btilde^\intercal Q+q^\intercal) A^k y_0 +\btilde^\intercal Q \btilde+q^\intercal \btilde
\end{equation}
We conclude that we can use the results developed in Subsections~\ref{bigksenu}--\ref{algo} where the matrix $Q$ is unchanged and the vector $q$ becomes $2Q\btilde+q$. The polytope of initial conditions also changes since we have to consider now $\xin-\btilde$. Note that $\mathcal{E}(\xin-\btilde)=\mathcal{E}(\xin)-\btilde$.

\section{Implementation and Experiments}
\label{experiments}

%

\subsection{Example}
We illustrate our techniques on one academic example. The example deals with a linear system and convex objective functions.
 
\comment{
\subsubsection{With a linear system}
}
We consider the discretisation of an harmonic oscillator $\ddot{x}+\dot{x}+x=0$ by an explicit Euler scheme. The discretization step is set to 0.01. Introducing the position variable, $x$ and the speed variable $v$. We assume that the initial conditions can be taken into the set $[-1,1]^2$. The Euler scheme becomes a linear discrete-time system in dimension two defined as follows: 
\begin{equation}
\label{harmonic}
\begin{pmatrix}
x_{k+1}\\
v_{k+1}
\end{pmatrix}
=\begin{pmatrix}
1 & 0.01\\
-0.01 & 0.99
\end{pmatrix}
\begin{pmatrix}
x_{k}\\
v_{k}
\end{pmatrix},\ (x_0,v_0)\in [-1,1]^2
\end{equation}

\subsubsection*{Homogeneous convex objective functions}

For this linear system, we are interested in computing:
\begin{itemize}
\item the maximal value of the Euclidean norm of the state-variable $\norm{(x_k,v_k)}_2^2$ and thus $Q=\begin{pmatrix}
1 & 0\\ 
0 & 1
\end{pmatrix}$ ;
\item the square of the position variable $x_{k}^2$ and thus $Q=\begin{pmatrix}
1 & 0\\ 
0 & 0
\end{pmatrix}$;
\item the speed variable $v_{k}^2$ and thus $Q=\begin{pmatrix}
0 & 0\\ 
0 & 1
\end{pmatrix}$.
\end{itemize}
The associated quadratic objective functions are homogeneous and convex.
\comment{
U=\dfrac{1}{2}\begin{pmatrix}
\sqrt{2} & \sqrt{2}\\
\dfrac{i\sqrt{6}-\sqrt{2}}{2} & -\dfrac{i\sqrt{6}+\sqrt{2}}{2}
\end{pmatrix}, 
\text{ and }
U^{-1}=\begin{pmatrix}
\dfrac{\sqrt{18}-i\sqrt{6}}{6} & -\dfrac{i\sqrt{6}}{3}\\
\dfrac{\sqrt{18}+i\sqrt{6}}{6} & \dfrac{i\sqrt{6}}{3}
\end{pmatrix}
}
The matrix $A$ is diagonalizable and we can take:
\[
U=\begin{pmatrix}
1 & 1\\
\dfrac{i\sqrt{3}-1}{2} & -\dfrac{i\sqrt{3}+1}{2}
\end{pmatrix} \text{ and }
D=\begin{pmatrix} 
\dfrac{199+i\sqrt{3}}{200} & 0\\ 0 & 
\dfrac{199-i\sqrt{3}}{200}
\end{pmatrix}
\]
We conclude that $\rho(A)=\dfrac{\sqrt{9901}}{100}<1$.
\comment{
$U^\star=\dfrac{1}{4} 
\begin{pmatrix}
4&1+i\sqrt{3} \\
1+i\sqrt{3} & 4
\end{pmatrix}
$.
}

\noindent To compute the maximal value of any convex/concave quadratic objective functions on the reachable values of system~\eqref{harmonic}, we need the following information: 
$U^\star=\begin{psmallmatrix}
1 & & -\frac{i\sqrt{3}+1}{2}\\
1 & &\frac{i\sqrt{3}-1}{2}
\end{psmallmatrix}
$ and $\mu((UU^\star)^{-1})=\displaystyle{\sup_{(x,y)\in[-1,1]^2}} 3^{-1}(2x^2+2y^2+2xy)=2$.

We need supplementary values depending on the objective quadratic function. We give those values for each three problems presented earlier:
\begin{itemize}
\item For $Q=\Idd_2$, we have $\lmax{U^\star Q U}=\lmax{U^\star U}=3$;
and $\nu_0=\sup_{(x,y)\in [-1,1]^2} x^2+y^2=2$. As $\sqrt{|\lmax{U^\star Q U}|\mu((UU^\star)^{-1})}=\sqrt{6}$, the hypothesis of Corollary~\ref{corollary1} i.e. $\nu_0\geq \sqrt{|\lmax{U^\star Q U}|\mu((UU^\star)^{-1})}$ does not hold. As $\nu_0>0$, we have $\ks_\nu=0$ and : 
\[
\kdiag_\nu(0)=\left\lfloor\dfrac{\ln\left(\sqrt{\nu_0/|\lmax{U^\star Q U}|\mu((UU^\star)^{-1})}\right)}{\ln(\rho(A))}\right\rfloor+1=\left\lfloor\dfrac{\ln(\sqrt{1/3})}{\ln(\rho(A))}\right\rfloor+1=111
\]
To compute the optimal value and a maximizer, we compute $\nu_k$ for all $k=0,\ldots,\kdiag_\nu(0)$. By doing so, we find that the optimal value $\nuopt$ is equal to 2 reached at $\kopt$ equal to 0 for the vertex $\xyopt$ equal to $(1,1)^\intercal$. Note that in Algorithm~\ref{algoAffQP}, we should update $\kdiag_\nu$. In this example, as the optimal value is found at $k=0$, then $\kdiag_\nu$ does never change. 
  
\item For $Q=\begin{psmallmatrix}
 1 & 0\\
 0 & 0
 \end{psmallmatrix}$: $\lmax{U^\star Q U}=2$, $\nu_0=\sup_{(x,y)\in [-1,1]^2} x^2=1$ and $\sqrt{|\lmax{U^\star Q U}|\mu((UU^\star)^{-1})}=2$. Again the hypothesis of Corollary~\ref{corollary1} does not hold and $\ks_\nu=0$. Finally: 
\[
\kdiag_\nu(0)=\left\lfloor\dfrac{\ln(1/2)}{\ln(\rho(A))}\right\rfloor+1=140
\]
We, then, compute $\nu_k$ for all $k=0,\ldots,\kdiag_\nu(0)$. By doing so, we find $\nuopt\simeq 1.64886$, $\kopt=61$ and $\xyopt=(1,1)^\intercal$. In Algorithm~\ref{algoAffQP}, the integer $\kdiag_\nu$ is modified when the optimal value increases. In particular, since $\nu_1=1.21$, we get
$\kdiag_\nu(1)=121$. Then, we get $\kdiag_\nu(2)=119$,\ldots, $\kdiag_\nu(61)=90$. In consequence, we reduce our first estimate of number of iterations from 140 to 90. 

\item For $Q=\begin{psmallmatrix}
 0 & 0\\
 0 & 1
 \end{psmallmatrix}$: $\lmax{U^\star Q U}=2$, $\nu_0=\sup_{(x,y)\in [-1,1]^2} y^2=1$ and $\sqrt{|\lmax{U^\star Q U}|\mu((UU^\star)^{-1})}=2$. Again the hypothesis of Corollary~\ref{corollary1} does not hold and $\ks_\nu=0$. Finally: 
\[
\kdiag_\nu(0)=\left\lfloor\dfrac{\ln(1/2)}{\ln(\rho(A))}\right\rfloor+1=140
\]
We, then, compute $\nu_k$ for all $k=0,\ldots,\kdiag_\nu(0)$. By doing so, we find $\nuopt=1$, $\kopt=0$ and $\xyopt=(1,1)^\intercal$. As $\kopt=0$, then  $\kdiag_\nu$ never changes in Algorithm~\ref{algoAffQP}.
\end{itemize}

\noindent We remark that for the case where $Q=\begin{psmallmatrix}
 1 & 0\\
 0 & 0
 \end{psmallmatrix}$ and the case $Q=\begin{psmallmatrix}
 0 & 0\\
 0 & 1
 \end{psmallmatrix}$, $\kdiag_\nu(0)$ is the same. Indeed, for $P=\begin{psmallmatrix}
 0 & 1\\
 1 & 0
 \end{psmallmatrix}$, we have $\begin{psmallmatrix}
 0 & 0\\
 0 & 1
 \end{psmallmatrix}=P\begin{psmallmatrix}
 1 & 0\\
 0 & 0
 \end{psmallmatrix}P$. The matrix $P$ is an orthogonal permutation matrix. Thus, the eigenvalues do not change. The value $\nu_0$ is not impacted by this matrix multiplication. Indeed, the coordinates of any initial vector are just permuted and the intervals for the coordinates are the same. The final optimal value is only value affected by this modification. 

\noindent In this example, we note that $\ks_\nu=0$ following Prop.~\ref{naiveresult} as the interior of the initial set is non-empty and the objective function is homogeneous and convex.

\begin{remark}
In the preliminary work using Lyapunov function~\cite{adj2018optimal}, we had coarser integers $\kdiag_\nu$. For $Q=\Idd_2$, we had $\kdiag_\nu(0)=130$, for $Q=\begin{psmallmatrix}  1 & 0\\ 0 & 0
 \end{psmallmatrix}$, $\kdiag_\nu(0)=188$ and for $Q=\begin{psmallmatrix}  0 & 0\\ 0 & 1
 \end{psmallmatrix}$, $\kdiag_\nu(0)=221$.
\end{remark}
\subsubsection*{Non-Homogeneous convex objective function}
Let us consider the same linear system depicted at Eq~\eqref{harmonic}. 
We consider another optimization problem over the reachable values of the system. We are interested in the computation of the optimal value :
\[
\sup_{k\in\nn} x_k^\intercal Q x_k+q^\intercal x_k
\text{ where }Q=\begin{pmatrix} 1 & -1/2 \\ -1/2 & 1/4\end{pmatrix}\text{ and }q^\intercal=(-1,1/2)\enspace .
\]

We use the same spectral decomposition as before and thus, we still have $U^\star=\begin{psmallmatrix}
1 & & -\frac{i\sqrt{3}+1}{2}\\
1 & &\frac{i\sqrt{3}-1}{2}
\end{psmallmatrix}
$ and $\mu((UU^\star)^{-1})=\displaystyle{\sup_{(x,y)\in[-1,1]^2}} 3^{-1}(2x^2+2y^2+2xy)=2$. Since the objective function is not homogeneous, we have $\vdiag\neq 0$. Actually, we have $\norm{U^\star q}=\sqrt{7/2}$ and $|\lmax{U^\star Q U}|=7/2$. Hence, $\vdiag=1/2$.

We have $\nu_0=\sup_{(x,y)\in [-1,1]^2} x^2-xy+0.25y^2-x+0.5 y =15/4$ and $(\sqrt{|\lmax{U^\star Q U}|\mu((UU^\star)^{-1})}+\vdiag)^2-\vdiag^2=7+\sqrt{7}>\nu_0$ (Corollary~\ref{corollary1} does not hold). We compute 
\[
\kdiag_\nu(0)=\left\lfloor\dfrac{\ln\left((\sqrt{\nu_0+\vdiag^2}-\vdiag)/\sqrt{|\lmax{U^\star Q U}|\mu((UU^\star)^{-1})}\right)}{\ln(\rho(A))}\right\rfloor+1=\left\lfloor\dfrac{\ln\left(3/(2\sqrt{7})\right)}{\ln(\rho(A))}\right\rfloor+1=115
\]
By computing $\nu_k$ for $k=0,\ldots,115$, we found as optimal value $\nuopt=3.75$ at $\kopt=0$ and for $\xopt=(-1,1)^\intercal$. 

The structure of the objective function is particular. The matrix $Q$ is actually equal to $qq^\intercal$. This fact explains the value $\vdiag$. This particular case is motivated by some verification purpose. We found as optimal value $3.75$. It means that we have  
$(x_k-0.5 v_k)^2-x_k+0.5 v_k\leq 3.75$ or $-3/4 \leq x_k-0.5 v_k\leq 3$.
\subsection{Implementation and Benchmarks}
We implement Algorithm~\ref{algoAffQP} in Julia 1.4.0~\cite{bezanson2017julia} on a laptop equipped with a Intel(R) Core(TM) i5-6300U @ 2,40GHz processor and 8Gb RAM memory. To solve convex quadratic programs we use the solver IpOpt~\cite{wachter2006implementation}. The linear algebra tools such as eigendecomposition and eigenvalues extraction have been managed by the standard library LinearAlgebra of Julia. 





\subsubsection{Benchmarks protocol}
We generate 100 random benchmarks for the possible combinations of the problem : {linear/affine} systems, {convex/concave} and {homogeneous/non-homogeneous} objective functions. 

The protocol of the benchmarks is as follows. We generate randomly a matrix the spectral radius of which is strictly less than 1. Moreover, we regenerate a new matrix if it is not diagonalizable. If the system is affine, we also generate randomly a vector. To complete the definition of discrete-time system, we need an initial set. In the convex case, we use a vertex representation of the initial set and generate randomly a certain number of vertices. In the concave case, we use a constraints representation of the initial set. In our benchmarks, we restrict ourselves initial sets to be boxes. Our code guarantees the non-empty box. Note that, even for the concave case, we need the vertices of the initial polytopic set. Indeed, we compute the maximum of the convex function $x {UU^\star}^{-1}x$ on $\xin$ using Lemma~\ref{lemma2}.  For the objective function, a symmetric matrix is generated. This matrix can be positive semi-definite or negative definite depending on the test (convex/concave). If we need a non-homogeneous objective quadratic function, we generate randomly a vector. The algorithm returns a status:
\begin{itemize}
    \item Failed, if a positive term for the sequence has not been reached before the maximal number parameter $N$;
    \item Corollary 1, if $\nu_0\geq \left(\sqrt{|\lmax{U^*QU}|\mus{(UU^*)^{-1}}}+\vdiag\right)^2-\vdiag^2$ holds;
    \item $\kdiag_\nu$ if in Algorithm~\ref{algoAffQP}, Line 14 is reached;
\end{itemize}

For each class of problems, we write the type of objective function (Obj. Type). The following abbreviations are used in Tables~\ref{Lineartable} and~\ref{Affinetable}: 
\begin{itemize}
\item CXH for convex and homogeneous;
\item CXnH for convex and non-homogeneous;
\item CAH for concave and non-homogeneous;
\item CAnH for concave and non-homogeneous.
\end{itemize}
Tables~\ref{Lineartable} and~\ref{Affinetable}, we also write the dimension (Dim.) of the system; the number of vertices (Ver. Nb.); the number of occurrences of each status (Status C=Corollary 1;K=$\kdiag_\nu$; F=Failed); the average of resolution time (Avg. Time) in seconds ; the average of memory used (Avg. Mem.) in Megabits (MiB) and the maximal memory size used (Mx. Mem.). We complete the benchmark tables with the average of the ranks of the first positive term of the sequence $(\nu_k)_k$ (Avg. $\ks_{\nu}$) and the maximum of this rank over the 100 generated instances (Mx. $\ks_\nu$). Next, we present the average of the number of iterations (Avg. It. Nb.) made to solve the problem. This number of iterations is actually provided by $\kdiag_\nu(\bigkse_\nu)$ i.e. the last $\kdiag_\nu$ computed at Line 21 of Algorithm~\ref{algoAffQP}. The maximum over all numbers of iterations is also provided (Mx. It. Nb.). The two last columns of Tables~\ref{Lineartable} and~\ref{Affinetable} concern the difference between the number of iterations and the rank $\bigkse_\nu$ i.e. the smallest maximizer rank. One column gives the average over those differences (Avg. It.-Opt.) and the other the maximum of all differences (Mx. It.-Opt.). 

The results of the benchmarks for the linear systems are presented at Table~\ref{Lineartable} whereas the ones for affine systems are presented at Table~\ref{Affinetable}. 

\renewcommand{\arraystretch}{1.3}
\begin{table}
\centering
\begin{tabular}{|*{12}{c|}}
\hline
\hline
Obj. & Dim. & Ver. & Status & Avg. & Avg. & Avg. & Mx. & Avg. & Mx. & Avg. & Mx.\\
Type &      & Nb. & C/K/F & Time  & Mem. & $\ks_\nu$ & $\ks_\nu$& It. Nb. & It. Nb. & It.-Opt. & It.-Opt.\\
\hline
\hline
CXH & 2 & 100 & 0/100/0 & 0.0004 & 0.185 & 0 & 0 & 4 & 19 & 3 & 19 \\
CXH & 2 & 1000 & 0/100/0 & 0.0029 & 1.761 & 0 & 0 & 4 & 16 & 4 & 16 \\
CXH & 2 & 100000 & 0/100/0 & 0.3153 & 197.324 & 0 & 0 & 5 & 23 & 4 & 16 \\
CXH & 5 & 100 & 0/100/0 & 0.0017 & 1.074 & 0 & 0 & 25 & 124 & 23 & 124 \\
CXH & 5 & 1000 & 0/100/0 & 0.0189 & 10.304 & 0 & 0 & 26 & 67 & 24 & 67 \\
CXH & 5 & 100000 & 0/100/0 & 1.5401 & 1030.766 & 0 & 0 & 26 & 90 & 24 & 89 \\
CXH & 10 & 100 & 0/100/0 & 0.0054 & 3.023 & 0 & 0 & 60 & 143 & 58 & 143 \\
CXH & 10 & 1000 & 0/100/0 & 0.0445 & 27.241 & 0 & 0 & 59 & 125 & 57 & 124 \\
CXH & 10 & 100000 & 0/100/0 & 4.0076 & 2894.589 & 0 & 0 & 64 & 188 & 62 & 188 \\
CXH & 20 & 100 & 0/100/0 & 0.0126 & 8.836 & 0 & 0 & 117 & 270 & 113 & 262 \\
CXH & 20 & 1000 & 0/100/0 & 0.1112 & 75.21 & 0 & 0 & 119 & 231 & 116 & 231 \\
CXH & 20 & 100000 & 0/100/0 & 9.5337 & 7695.709 & 0 & 0 & 125 & 262 & 121 & 262 \\
\hline
CXnH & 2 & 100 & 0/100/0 & 0.0004 & 0.201 & 0 & 0 & 4 & 16 & 4 & 16 \\
CXnH & 2 & 1000 & 0/100/0 & 0.003 & 1.916 & 0 & 0 & 5 & 17 & 4 & 14 \\
CXnH & 2 & 100000 & 0/100/0 & 0.2511 & 170.894 & 0 & 0 & 4 & 22 & 3 & 19 \\
CXnH & 5 & 100 & 0/100/0 & 0.0022 & 1.188 & 0 & 0 & 28 & 90 & 26 & 90 \\
CXnH & 5 & 1000 & 0/100/0 & 0.0182 & 11.341 & 0 & 0 & 28 & 127 & 27 & 127 \\
CXnH & 5 & 100000 & 0/100/0 & 1.5231 & 1093.95 & 0 & 0 & 28 & 97 & 26 & 97 \\
CXnH & 10 & 100 & 0/100/0 & 0.0053 & 3.144 & 0 & 0 & 62 & 164 & 60 & 164 \\
CXnH & 10 & 1000 & 0/100/0 & 0.0465 & 28.618 & 0 & 0 & 63 & 182 & 60 & 182 \\
CXnH & 10 & 100000 & 0/100/0 & 4.1589 & 2973.149 & 0 & 0 & 66 & 171 & 63 & 171 \\
CXnH & 20 & 100 & 0/100/0 & 0.0127 & 8.597 & 0 & 0 & 114 & 225 & 110 & 225 \\
CXnH & 20 & 1000 & 0/100/0 & 0.114 & 77.509 & 0 & 0 & 123 & 247 & 119 & 247 \\
CXnH & 20 & 100000 & 0/100/0 & 9.1668 & 7374.275 & 0 & 0 & 120 & 281 & 116 & 274 \\
\hline
CAnH & 2 & 4 & 0/95/5 & 0.2105 & 6.434 & 11 & 94 & 51 & 245 & 35 & 224 \\
CAnH & 5 & 32 & 0/90/10 & 0.7695 & 25.422 & 21 & 99 & 158 & 585 & 127 & 541 \\
CAnH & 7 & 128 & 0/85/15 & 1.0993 & 38.908 & 20 & 99 & 220 & 425 & 189 & 320 \\
CAnH & 10 & 1024 & 0/84/16 & 1.4928 & 68.479 & 23 & 98 & 290 & 499 & 259 & 488 \\
CAnH & 15 & 32768 & 0/86/14 & 2.3925 & 177.329 & 25 & 100 & 401 & 636 & 367 & 600 \\
CAnH & 20 & 1048576 & 0/89/11 & 4.9895 & 1512.594 & 25 & 90 & 501 & 713 & 468 & 658 \\
\hline
\hline
\end{tabular}
\caption{Experiments Table for Linear Systems}
\label{Lineartable}
\end{table}
\subsubsection{Discussions about the obtained results}
\paragraph{About linear systems.}
In Table~\ref{Lineartable}, the method succeeds to solve Problem~\eqref{optpb} when the objective function is convex. Unsurprisingly, in the convex case, the memory and the time consumption of the method grow exponentially when the number of vertices increases. This is due to the use of Lemma~\ref{lemma2}. To pass from homogeneous to non-homogeneous objective quadratic function does not impact the resolution time and the quantity of memory used for the resolution. 

In Table~\ref{Lineartable}, still for the convex case, we see that the first positive term of the sequence is still achieved at the rank $k=0$. This means that our Prop.~\ref{naiveresult} could be extended to a more general setting. For a fixed dimension of the system, the number of iterations is constant with respect to the number of vertices. However, the number of iterations increases with the dimension of the system. The dependency between our formula provided at Eq.~\eqref{kdiagform} and the dimension of the system will be studied in future works. As the integer $\bigkse_\nu$ is very close to 0 for the convex case, the columns "Avg. It. Nb" and "Avg. It.-Opt." are quite similar.   

In Table~\ref{Lineartable}, when the objective function is concave, the case where the objective function is homogeneous is not presented. Indeed, the values $\nu_k$ are still negative since $Q$ is negative definite and from Prop.~\ref{posetdelta} and Prop.~\ref{argmax}, the optimal value is equal to 0 and never achieves. When the objective function is non-homogeneous, the method succeeds in 88\% of cases. This is explained by the fact that the linear part has to compensate the negativity of terms $x^\intercal {A^k}^\intercal Q A^k x$. This is highlighted by the column "Mx. $\ks_\nu$" where the integers $\ks_\nu$ are really close to our parameter $N$ i.e. the maximal number of iterations for the search of a positive term. About the resolution time and the memory used, this scales better than in the convex case. The numbers of iterations are significantly  bigger than in the convex case and so the number of convex quadratic programs to solve blows up. This good scalability comes from the theory (the use of interior-points methods) and the practice (the use of a large-scale solver).

\begin{table}
\centering
\begin{tabular}{|*{12}{c|}}
\hline\hline
Obj. & Dim. & Ver. & Status & Avg. & Avg. & Avg. & Mx. & Avg. & Mx. & Avg. & Mx.\\
Type &      & Nb. & C/K/F & Time  & Mem. & $\ks_\nu$ & $\ks_\nu$& It. Nb. & It. Nb. & It.-Opt. & It.-Opt.\\
\hline
\hline
CXH & 2 & 100 & 0/77/23 & 0.0022 & 1.126 & 2 & 60 & 17 & 154 & 13 & 84 \\
CXH & 2 & 1000 & 0/66/24 & 0.0227 & 14.132 & 5 & 92 & 22 & 192 & 15 & 94 \\
CXH & 2 & 100000 & 0/69/31 & 1.9327 & 1174.1 & 2 & 60 & 14 & 123 & 10 & 66 \\
CXH & 5 & 100 & 0/76/24 & 0.0043 & 2.548 & 5 & 79 & 56 & 202 & 47 & 121 \\
CXH & 5 & 1000 & 0/72/28 & 0.0363 & 23.548 & 5 & 68 & 49 & 228 & 40 & 142 \\
CXH & 5 & 100000 & 0/78/22 & 3.4679 & 2098.026 & 4 & 41 & 47 & 211 & 41 & 162 \\
CXH & 10 & 100 & 0/82/18 & 0.0091 & 4.609 & 6 & 89 & 93 & 352 & 83 & 243 \\
CXH & 10 & 1000 & 0/81/19 & 0.0908 & 43.871 & 6 & 57 & 97 & 246 & 87 & 198 \\
CXH & 10 & 100000 & 0/78/22 & 7.5454 & 4298.45 & 5 & 55 & 104 & 277 & 95 & 257 \\
CXH & 20 & 100 & 0/73/27 & 0.0195 & 11.173 & 8 & 91 & 170 & 564 & 157 & 482 \\
CXH & 20 & 1000 & 0/78/22 & 0.1587 & 90.585 & 10 & 79 & 170 & 364 & 154 & 341 \\
CXH & 20 & 100000 & 0/82/18 & 14.8207 & 8758.131 & 8 & 88 & 154 & 397 & 141 & 292 \\
\hline
CXnH & 2 & 100 & 0/75/25 & 0.0026 & 1.054 & 1 & 6 & 11 & 107 & 9 & 106 \\
CXnH & 2 & 1000 & 0/71/29 & 0.0246 & 11.746 & 2 & 28 & 14 & 82 & 11 & 54 \\
CXnH & 2 & 100000 & 0/79/21 & 1.6282 & 979.075 & 2 & 34 & 16 & 114 & 12 & 69 \\
CXnH & 5 & 100 & 0/68/32 & 0.0059 & 2.686 & 6 & 96 & 57 & 256 & 48 & 146 \\
CXnH & 5 & 1000 & 0/68/32 & 0.0463 & 23.483 & 4 & 56 & 50 & 223 & 42 & 182 \\
CXnH & 5 & 100000 & 0/75/25 & 3.5164 & 2125.437 & 3 & 52 & 45 & 200 & 38 & 145 \\
CXnH & 10 & 100 & 0/64/36 & 0.0102 & 4.727 & 8 & 54 & 98 & 247 & 84 & 224 \\
CXnH & 10 & 1000 & 0/74/26 & 0.0784 & 40.883 & 6 & 81 & 91 & 248 & 81 & 208 \\
CXnH & 10 & 100000 & 0/78/22 & 7.0975 & 4270.664 & 8 & 86 & 103 & 267 & 90 & 207 \\
CXnH & 20 & 100 & 0/74/26 & 0.0184 & 10.692 & 8 & 100 & 162 & 531 & 149 & 415 \\
CXnH & 20 & 1000 & 0/81/19 & 0.1707 & 99.201 & 7 & 74 & 172 & 381 & 159 & 361 \\
CXnH & 20 & 100000 & 0/75/25 & 14.6931 & 8848.795 & 9 & 85 & 166 & 397 & 151 & 331 \\
\hline
CAH & 2 & 4 & 0/100/0 & 0.141 & 3.251 & 1 & 53 & 25 & 106 & 23 & 106 \\
CAH & 5 & 32 & 0/100/0 & 0.3965 & 7.687 & 1 & 9 & 45 & 128 & 44 & 126 \\
CAH & 7 & 128 & 0/100/0 & 0.5147 & 11.554 & 0 & 0 & 56 & 155 & 56 & 155 \\
CAH & 10 & 1024 & 0/100/0 & 0.6534 & 20.191 & 0 & 0 & 71 & 177 & 71 & 177 \\
CAH & 15 & 32768 & 0/100/0 & 1.0462 & 75.806 & 0 & 0 & 90 & 213 & 90 & 213 \\
CAH & 20 & 1048576 & 0/100/0 & 3.552 & 1466.55 & 0 & 0 & 110 & 284 & 110 & 284 \\
\hline
CAnH & 2 & 4 & 0/100/0 & 0.1354 & 3.092 & 1 & 12 & 24 & 93 & 22 & 57 \\
CAnH & 5 & 32 & 0/100/0 & 0.3423 & 7.59 & 1 & 5 & 45 & 124 & 44 & 124 \\
CAnH & 7 & 128 & 0/100/0 & 0.5054 & 10.633 & 0 & 0 & 52 & 150 & 51 & 150 \\
CAnH & 10 & 1024 & 0/100/0 & 0.6846 & 19.402 & 0 & 0 & 68 & 191 & 68 & 189 \\
CAnH & 15 & 32768 & 0/100/0 & 1.0987 & 77.652 & 0 & 0 & 95 & 220 & 95 & 220 \\
CAnH & 20 & 1048576 & 0/100/0 & 3.4749 & 1465.129 & 0 & 0 & 108 & 233 & 108 & 233 \\
\hline\hline
\end{tabular}
\caption{Experiments Table for Affine Systems}
\label{Affinetable}
\end{table}
\paragraph{About affine systems.}
The resolution time and the memory used to solve one instance are a little bit higher in Table~\ref{Affinetable} for the convex case compared with the ones for linear systems of Table~\ref{Lineartable}. This can be simply explained by the increase of the number of iterations. The reformulation proposed at Subsection~\ref{affine} of  Problem~\eqref{optpb} for a pure affine system keeps the same structure (the same dimension and the same number of extreme points/or constraints). The only change is about the homogeneity of the objective function. Homogeneity only appears in very particular situation where $2\btilde^\intercal Q+q^\intercal =0$. Thus, even for objective functions which are homogeneous initially, the obtained reformulation has, in general, a non-homogeneous objective function.   
Furthermore, we see at Table~\ref{Affinetable} that, for the convex case, the method succeeds to solve Problem~\eqref{optpb} is 74.75\% of cases. Recall that a failure happens when a positive term has not been found before $N$ steps. At Equation~\eqref{affineequation}, since $y_0$ lies in $\xin-\btilde$ we have, for $k=0$, $(2\btilde Q+q^\intercal) A^k y_0=(2\btilde Q+q^\intercal ) x_0-2\btilde Q \btilde +q^\intercal \btilde$. We, thus, add a negative term $-2\btilde Q \btilde$ as $Q$ is positive semi-definite. In contrary, when $Q$ is negative definite this term is positive and might help the objective function to be positive. 

We remark that in both linear and affine cases, the status "Corollary 1" never happens. It should be interesting to investigate if this inequality can hold or not. 




\section{Conclusion And Future Works}
\label{conclusion}
In this paper, we develop a method to solve the problem of maximizing a convex or concave quadratic function over the reachable values set of a convergent affine discrete-time system. This is the same as searching the term of the sequence defined by the system for which the objective is maximal. The method proposed in the paper consists in finding the smallest possible rank of the sequence for which the search of the maximum is useless. Actually, we construct a family of integers which are parameterized by the positive terms of the sequence. Then if a new positive term is given, then we can recompute a new rank. This idea is used inside the algorithm to reduce the number of iterations. 

The results of the prototype implemented in Julia on a personal laptop are promising. Nevertheless, some computational aspects might be improved. The computation of the maximum of a convex quadratic over a polytope could use more scalable methods~\cite{konno1976maximization,Floudas1995,Tuy2016}. Moreover, the indefinite case might be considered in the same time. 


In this paper, we have only considered diagonalizable matrix dynamics. The non-diagonalizable case has to be studied. A direct decomposition of the matrix dynamics such as Dunford decomposition does not permit to construct an uniform integer $\kdiag_\nu$ as it is done here. The use of Lyapunov functions~\cite{adj2018optimal} must generalize the approach even if it seems to be less precise in practice. 

Finally, the future directions of research must include more general dynamics and more general objective functions. We may think about piecewise affine or more generally piecewise polynomial systems and polynomial objective functions.

\section{Acknowledgements}
The author would like to Milan Korda for his suggestions about the paper.
\bibliographystyle{plain}
\bibliography{paperbib-final}
\end{document}